\documentclass{amsart}

\usepackage{setspace}
\usepackage{graphicx}
\usepackage{verbatim}   
\usepackage{color}      
\usepackage{subfigure}  
\usepackage{hyperref}   

\newtheorem{thm}{Theorem}[section]

\theoremstyle{definition}

\newtheorem{ex}[thm]{Example}

\numberwithin{equation}{section}
\def\BMT{\lower0.9em\hbox{\includegraphics{BMT.pdf}}}
\def\TMB{\lower0.9em\hbox{\includegraphics{TMB.pdf}}}
\def\Msplit{\lower0.8em\hbox{\includegraphics{Msplit.pdf}}}
\def\Tsplit{\lower1.0em\hbox{\includegraphics{Tsplit.pdf}}}
\def\Bsplit{\lower1.0em\hbox{\includegraphics{Bsplit.pdf}}}
\def\NewTsplit{\lower1.0em\hbox{\includegraphics{Tsplit2.pdf}}}
\def\NewBsplit{\lower1.0em\hbox{\includegraphics{Bsplit2.pdf}}}
\def\Ufirstsplit{\lower0.7em\hbox{\includegraphics{Ufirstsplit.pdf}}}
\def\Usecondsplit{\lower0.7em\hbox{\includegraphics{Usecondsplit.pdf}}}

\begin{document}
\title[Bipyramids and Bounds on Volumes of Hyperbolic Links]
{Bipyramids and Bounds on Volumes of Hyperbolic Links}

\date{\today}
\author[Colin Adams]{Colin Adams}
\address{Department of Mathematics and Statistics, Williams College, Williamstown, MA 01267}
\email{Colin.C.Adams@williams.edu}

\begin{abstract}We utilize ideal bipyramids to obtain new upper bounds on volume for hyperbolic link complements in terms of the combinatorics of their projections. \end{abstract}
\maketitle

\section{Introduction}\label{S:intro}  Knots fall into three disjoint categories: torus knots, satellite knots (including composite knots) and hyperbolic knots. A knot is hyperbolic if its complement $S^3-K$ carries a hyperbolic metric. Such a metric is uniquely determined and hence, the hyperbolic volume of its complement becomes an invariant that can be used to distinguish it from other knots. 

 Given a projection $P$ of a knot, we will say it is {\it reduced} if there are no Type I or Type II Reidemeister moves that would individually lower the number of crossings, or a circle in the projection plane that intersects the knot in only one crossing, with two strands to either side of it. In a reduced projection, define  a {\it maximal bigon chain} to be a sequence of bigonal complementary regions touching end-to-end that is as long as possible. The {\it crossing length} of such a bigon chain is the number of crossings it contains. Note that a crossing that does not lie on a bigon is considered to be a maximal bigon chain of crossing length 1. The number of such maximal bigon chains in a projection $P$ is called the {\it twist number} of the projection, denoted $t(P)$. We further let $t_i(P)$ be the number of maximal bigon chains of crossing length $i$ and $g_i(P)$ be the number of maximal bigon chains of crossing number at least $i$. We typically assume that our diagrams are twist reduced, which mean flypes have been applied to minimize the number of distinct bigon chains.

In a variety of papers, there have been attempts to determine bounds on hyperbolic volume from projections of knots and links. In \cite{Adams0} it was proved that with the exception of the figure-eight knot, a hyperbolic knot satisfies $\mbox{vol}(S^3-K) \leq (4c-16)v_{tet}$, where $c$ is the crossing number and $v_{tet}$ is the volume of an ideal regular tetrahedron, approximately $1.01494$. As it is obtained by decomposing the link complement into ideal tetrahedra, we call this the tetrahedral bound.

In an appendix by Ian Agol and Dylan Thurston to \cite{Lack}, it was proved that if a hyperbolic knot or link $L$ has  a reduced alternating projection $P$, then $\mbox{vol}(S^3 -L) \leq 10 v_{tet} (t(P)-1)$. We refer to this as the AT bound on volume. 

In \cite{DT}, Dasbach and Tsvietkova refined this bound so that if $L$ is a hyperbolic alternating link in a reduced alternating projection $P$, then $\mbox{vol}(S^3-L) \leq(4t_1(P) + 6t_2(P) + 8t_3(P) + 10g_4(P) -a)v_{tet}$ where $a=10$ when $g_4 \ge 1 $, $a = 7$ when $g_4= 0$ but $t_3 \ge 1$ and $a=6$ otherwise. We refer to this as the DT bound on volume.

There is another approach to bounding volume that follows from a construction of  D. Thurston. One places an octahedron at each crossing with its top vertex on the bottom of the overstrand and its bottom vertex on the top of the understrand as in Figure \ref{octahedron}. 

\begin{figure}[h]
\begin{center}
\includegraphics[scale=0.7]{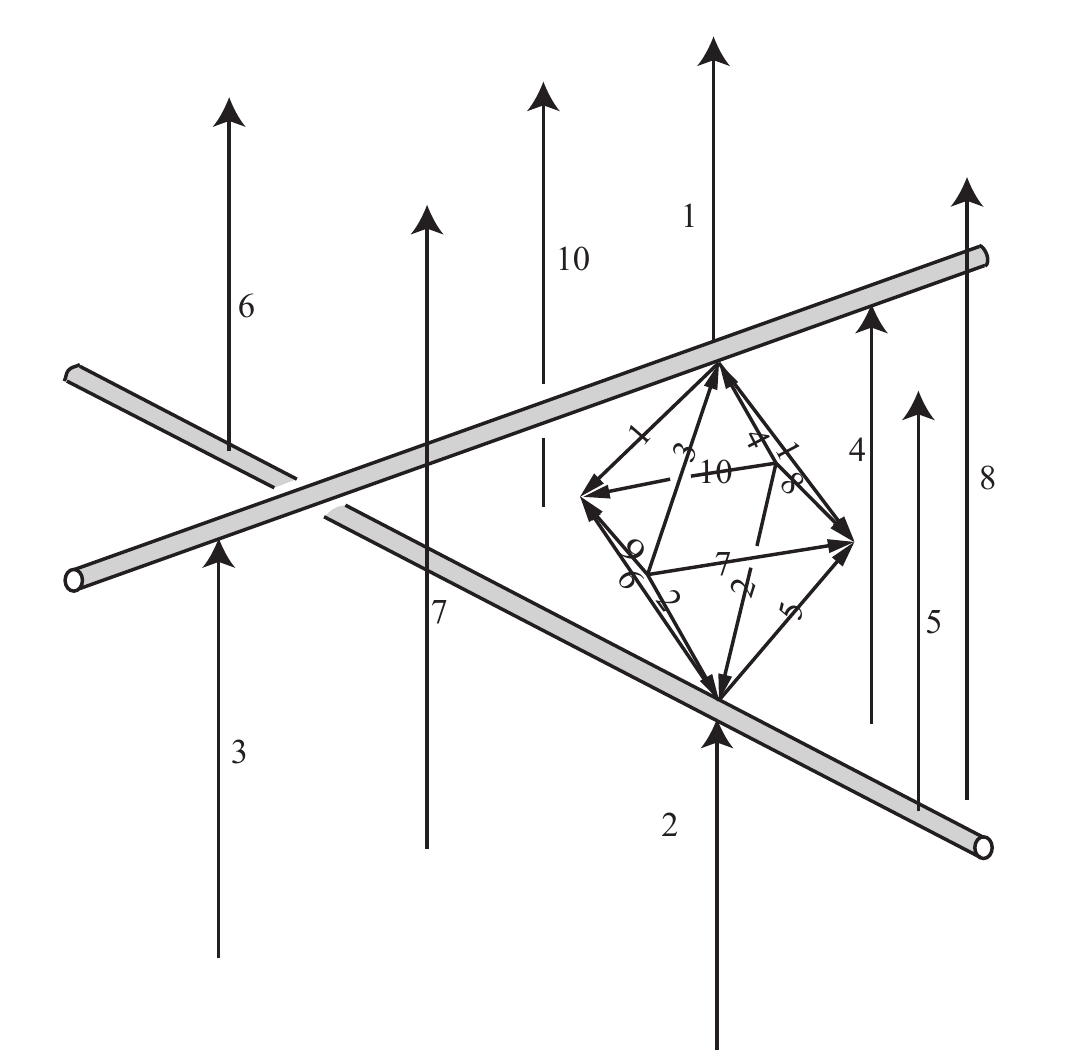}
\caption{Placing an octahedra between each crossing.}
\label{octahedron}
\end{center}
\end{figure}

Then, as per the edge labelings in Figure \ref{octahedron}, one pulls two of the opposite remaining vertices up to meet above the crossing at a point denoted $U$, thereby identifying two edges on the octahedron, and pulls the remaining two vertices down to meet below the crossing at a point denoted $D$, again identifying two edges of the octahedron.Then one can glue together the faces on the octahedra at the various crossings in order to fill the complement of the knot. If the two vertices at $U$ and $D$ are included, we have a decomposition of the knot complement into octahedra with some finite vertices and some ideal vertices. This construction has been used to attempt to prove Kaeshaev's Volume Conjecture for various categories of knots. See \cite{Mur} and \cite{Yok} for instance. However, one also obtains an upper bound for the hyperbolic volume of the knot or link complement since  an octahedron in hyperbolic 3-space has  volume at most the volume of an ideal regular octahedron, which is $v_{oct} \approx 3.6638$. So $\mbox{vol}(S^3 -L) \leq  c  v_{oct}$. This bound was improved in \cite{Adams5} by pulling the finite vertices to ideal vertices.

\begin{thm}\label{volumethm} Let $L$ be a hyperbolic knot or link with $c \geq 5$ crossings. Then \\ $vol(S^3 -L) \leq  (c - 5) v_{oct} + 4 v_{tet}$.
\end{thm}

We call this the octahedral bound.

In this paper, we utilize Thurston's octahedral construction to generate several other constructions in terms of bipyramids, one based on bigon chains and one based on faces of the projection. Both give improved upper bounds on volume. For instance, the first construction, which is described in Section 3, yields the following improvement on the DT bound. 

\newtheorem*{thm:bigonvolume}{Theorem \ref{thm:bigonvolume}}
\begin{thm:bigonvolume} Let $L$ be a  hyperbolic alternating link with five or more crossings in a reduced twist reduced alternating diagram that has at least three twist regions and that is not the Borromean rings. Then $\mbox{vol}(S^3-L) < t_1 v_{oct} +t_2(6 v_{tet}) + t_3 (7.8549\dots) + t_4(9.2375\dots) + g_5 (10 v_{tet}) - a$, where if $g_2 = 0$, then $a= 15.4972$, if $g_3 = 0$ but $t_2 \ge 1$, then $a = 11 v_{tet}$, if $g_4 = 0$ but $t_3 \ge 1$, then $ a =10.088$,  if $g_5 = 0$ but $t_4 \ge 1$, then, $a = 10.2873$ and if $g_5\ge 1$, then $a = 12.111$. 
\end{thm:bigonvolume}

This bigon chain bipyramid bound is called the BCB bound.

The second construction, which is described in Section 4, yields a means of obtaining a bound on volume in terms of the faces of the projection of any hyperbolic link. This face-centered bipyramid bound is called the FCB bound. In the cases of the $4_1$ and $6_2^2$, it yields the exact volume. For small volume knots, it appears to yield the smallest upper bound on volume of the various bounds discussed here. In Section 5, we compare the different bounds.

Note that various authors, in addition to considering upper bounds either in general or for specific classes of knots and links, have also determined lower bounds. See for instance \cite{FG}, \cite{FKP} and  \cite{Lack}.

Additional applications of the bipyramid construction appear in \cite{Adams7}, where it is used to obtain volume density results for 2-bridge links and \cite{Adams8}, where it is generalized and used to obtain volume density results for links in $S$ x I where $S$ is a surface.

\section*{Acknowledgement} Thanks to Aaron Calderon, Alexander Kastner, Xinyi Jiang, Gregory Kehne, and Nathaniel Mayer, for very helpful discussions and Mia Smith for the same and particularly for pointing out improvements to Theorem 2.2. Thanks also to NSF Grant DMS-1347804, and Williams College, which fund the SMALL REU program at Williams College.

\section{Ideal Bypyramids}

An ideal $n$-bypyramid $B_n$ appears as in Figure \ref{bipyramid}. We call such an ideal hyperbolic bipyramid {\it regular} if if can be constructed by gluing together $n$ ideal tetrahedra along a central edge, each isometric to the  ideal tetrahedron $T_n = T(2\pi/n, \frac{(n- 2)pi}{2n}, \frac{(n- 2)\pi}{2n})$, where each tetrahedron intersects the central edge with an edge of diherdral angle $2\pi/n$. Notice that all of the vertical edges of the resulting bipyramid have dihedral angle $\frac{(n- 2)\pi}{n}$, as do the edges going to the central vertex, while the edges connecting vertices that are neither the central vertex nor the vertex at $\infty$ have dihedral angle $2\pi/n$. The volume of $T_n$ is given by: $$\mbox{vol}(T_n) = \int_0^{2\pi/n} -2 ln(\sin \theta) \, d\theta + 2 \int_0^{\pi(n-1)/2n} -2 ln(\sin \theta) \, d\theta $$

\begin{figure}[h]
\begin{center}
\includegraphics[scale=0.5]{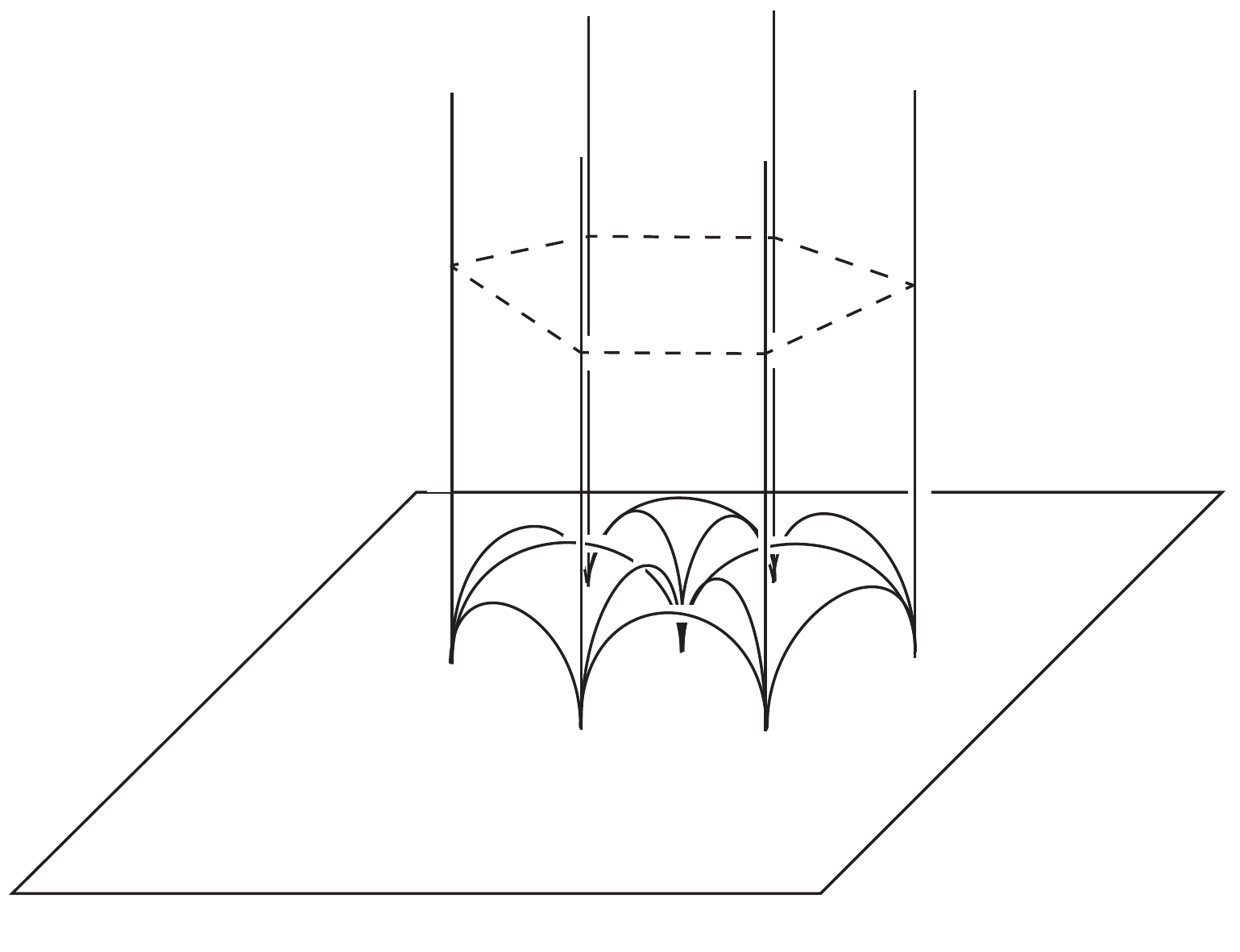}
\caption{An $n$-bipyramid for $n=6$.}
\label{bipyramid}
\end{center}
\end{figure}

\begin {thm} The maximal volume for an ideal $n$-bipyramid is obtained uniquely by the regular $n$-bipyramid.
\end{thm}

\begin{proof} An $n$-bipyramid can always be decomposed into n ideal tetrahedra, meeting at the central edge. Let them be denoted $T_1, T_2, \dots T_n$, where $T_i$ has angles $\alpha_i$ at the central edge and then $\beta_i$ and $\gamma_i$, reading clockwise. Then $\alpha_1 + \alpha_2 + \dots + \alpha_n = 2 \pi$. Let $f$ be the volume of the $n$-bipyramid. Then $$f= \sum_{i=1}^n \left(\int_0^{\alpha_i} - \ln{(2 \sin \theta)} d \theta +   \int_0^{\beta_i} - \ln{(2 \sin \theta)} d \theta +   \int_0^{\gamma_i} - \ln{(2 \sin \theta)} d \theta\right)$$.

We use Lagrange Multipliers to maximize $f$ subject to the following constraints:

For $i = 1, \dots , n$:

$$g_i = \alpha_i + \beta_i + \gamma_i - \pi = 0$$

Additionally, $$g_{n+1} =\alpha_1 + \alpha_2 + \dots + \alpha_n - 2 \pi = 0$$

Then we generate the following equations:

For $i = 1, \dots , n$:

$$-\ln(2\sin \alpha_i)= \lambda_i + \mu, \hspace{.1in}
-\ln(2\sin \beta_i)= \lambda_i ,\hspace{.1in}
-\ln(2\sin \gamma_i)= \lambda_i $$

From the second and third equations, we see that $\beta_i = \gamma_i$ for all $i$.
From the first and second equations, we have that for $i = 1, \dots , n$:

$$\mu = \ln(2\sin \beta_i)-\ln(2\sin \alpha_i)$$

Thus, $$\frac{\sin \beta_1}{\sin \alpha_1} = \frac{\sin \beta_2}{\sin \alpha_2} = \dots = \frac{\sin \beta_n}{\sin \alpha_n} $$

Since, $\beta_ i = \frac{\pi - \alpha_i}{2}$, this implies that $\alpha_1 = \alpha_2 = \dots = \alpha_n$.
Hence, $\alpha_i = \frac{2\pi}{n}$ and $\beta_i = \gamma_i = \frac{n-2}{n} \pi$, yielding a regular $n$-bipyramid.
\end{proof}

Volumes for the $n$-bipyramids for certain values of $n$ are provided in Table \ref{bipyramidvol}. 
\begin{table}[!ht]
\begin{center} 
  \begin{tabular}{|| c | c || }
\hline
n    & Volume   \\ \hline
2   &    0    \\ \hline    
3   &    2.0298    \\ \hline      
4   &    3.6638    \\ \hline      
5   &    4.9867    \\ \hline      
6   &    6.0896    \\ \hline      
7   &    7.0325    \\ \hline      
8   &    7.8549    \\ \hline      
9   &    8.5836    \\ \hline      
10   &    9.2375    \\ \hline      
11   &    9.8304    \\ \hline      
12   &    10.3725    \\ \hline  
13   &    10.8719    \\ \hline      
14   &    11.3347   \\ \hline      
20   &    13.5668    \\ \hline      
100   &    23.6709    \\ \hline      
1,000   &   38.1382    \\ \hline 
1,000,000   &    81.5409    \\ \hline      
1,000,000,000   & 124.944      \\ \hline                   
  \end{tabular}
   \caption{Volumes of $n$-bypyramids.}
   \label{bipyramidvol}
\end{center}
\end{table}

Note that with the exception of $n=6$, the volume of the corresponding $n$-bipyramid is always less than the volume that would be obtained by taking the sum of the volumes of $n$ regular tetrahedra. It is this volume saving that we take advantage of.

As n increases, the volumes of the $n$-bipyramids increase without bound. We show that the volumes grow logarithmically. 

\begin{thm} $\mbox{vol}(B_n) < 2\pi \ln (n/2)$ for $n \geq 3$ and $\mbox{vol}(B_n)$ grows asymptotically like $2\pi \ln (n/2)$.
\end{thm}
\begin{proof}  We show the second fact first.

\begin{align}
 \lim_{n \to \infty} \frac{\mbox{vol}(B_n)}{\ln (n/2)} &  =  \lim_{n \to \infty} \frac{n\mbox{vol}(T_n)}{\ln (n/2)} = 
  \lim_{n \to \infty} \frac{\mbox{vol}(T_n)}{(\ln (n/2) )/n}\\
& = \lim_{n \to \infty} \frac{-\ln{\left(2 \sin{\left(\frac{2\pi}{n}\right)}\right)} \frac{-2\pi}{n^2} + -2 \ln{\left(2\sin{\left(\frac{(n-2)\pi}{2n}\right)}\right)} \frac{\pi}{n^2}}{\frac{1- \ln (n/2)}{n^2}} \\
&= \lim_{n \to \infty} \frac{2\pi \ln{\left(2\sin{\left(\frac{2\pi}{n}\right)}\right)} - 2 \pi \ln \left( 2 \sin {\left(\frac{(n-2)\pi}{2n}\right)} \right)}{1-\ln (n/2)} =  \lim_{n \to \infty} \frac{2\pi \ln{\left(2\sin{\left(\frac{2\pi}{n}\right)}\right)}}{1-\ln (n/2)}\\
&  = \lim_{n \to \infty}\frac{\frac{-2 \cos \frac{2 \pi}{n}}{2\sin\frac{2\pi}{n}} \frac{(2\pi)^2}{n^2}}{\frac{-1}{n}}= 2\pi
\end{align}

Hence, the volumes of the $B_n$ grow asymptotically like $f(n) = 2\pi \ln (n/2)$. 

To see that $\mbox{vol} (B_n) < 2\pi \ln (n/2)$, we rewrite this as $\mbox{vol}(T_n) < 2\pi \frac{\ln (n/2)}{n}$. Let $f(n) =\mbox{vol} (T_n)$ and let $g(n) = 2\pi \frac{\ln (n/2)}{n}$.   Since these two functions agree asymptotically and since the result does hold for $n=3$ and 4, if it were the case that $f(a) > g(a)$ for some $a$, then it would have to be true that $f'(b) < g'(b)$  for some $b > a$. So we will prove that $f'(x) > g'(x)$ for all $x \ge 4$.

$$ f'(n)= \frac{2\pi}{n^2} \ln \left(\frac{\sin\left(\frac{2\pi}{n}\right)}{\sin\left(\frac{(n-2)\pi}{2n}\right)} \right)
= \frac{2\pi}{n^2} \ln \left(\frac{\sin\left(\frac{2\pi}{n}\right)}{\cos\left(\frac{\pi}{n}\right)}\right) =  \frac{2\pi}{n^2} \ln \left(2 \sin \left(\frac{\pi}{n}\right)\right)$$

$$g'(n) = 2\pi \left(\frac{1 - \ln (n/2)}{n^2}\right)$$

Then if $f'(b) < g'(b)$, we have

$$ \ln \left(2 \sin \left(\frac{\pi}{b}\right)\right) < 1 - \ln b/2$$

$$ b  \sin \left(\frac{\pi}{b}\right) < e$$

However, at $b=4$,  $b \sin \left(\frac{\pi}{b}\right) \approx 2.828 > e$ and $b \sin \left(\frac{\pi}{b}\right)$ is an increasing function approaching $2 \pi$, so this inequality never holds for $b \ge 4$.
\end{proof}

Note that the proof can be altered to show the slightly stronger result that $\mbox{vol}(B_n) < 2\pi \ln ( n/2.1818)$.

\section{Conjoining Octahedra}

In this section we consider sequences of octahedra in the Thurston construction that lie at the crossings of a maximal bigon chain. Consider two octahedra at the crossings of a bigon as in Figure \ref{bigonoctahedra}. They are glued together along four faces on each, however, we will focus on the two shaded faces on each. These are glued together as squares with a diagonal and no rotation. Hence we can glue the two octahedra together along this pair of faces to obtain a 6-bipyramid. If there is a subsequent bigon in the maximal bigon chain, then it can be glued onto the 6-bipyramid to create an 8-bipyramid. Thus, the octahedra that make up a maximal bigon chain of crossing length $n$ can be consolidated into a single $(2n+2)$-bipyramid. Thus, if a reduced projection $P$ has twist number $t(P)$, then the complement can be decomposed into $t(P)$ bipyramids, each with two ideal vertices and the rest divided into two equivalence classes of finite vertices. 

\begin{figure}[h]
\begin{center}
\includegraphics[scale=0.7]{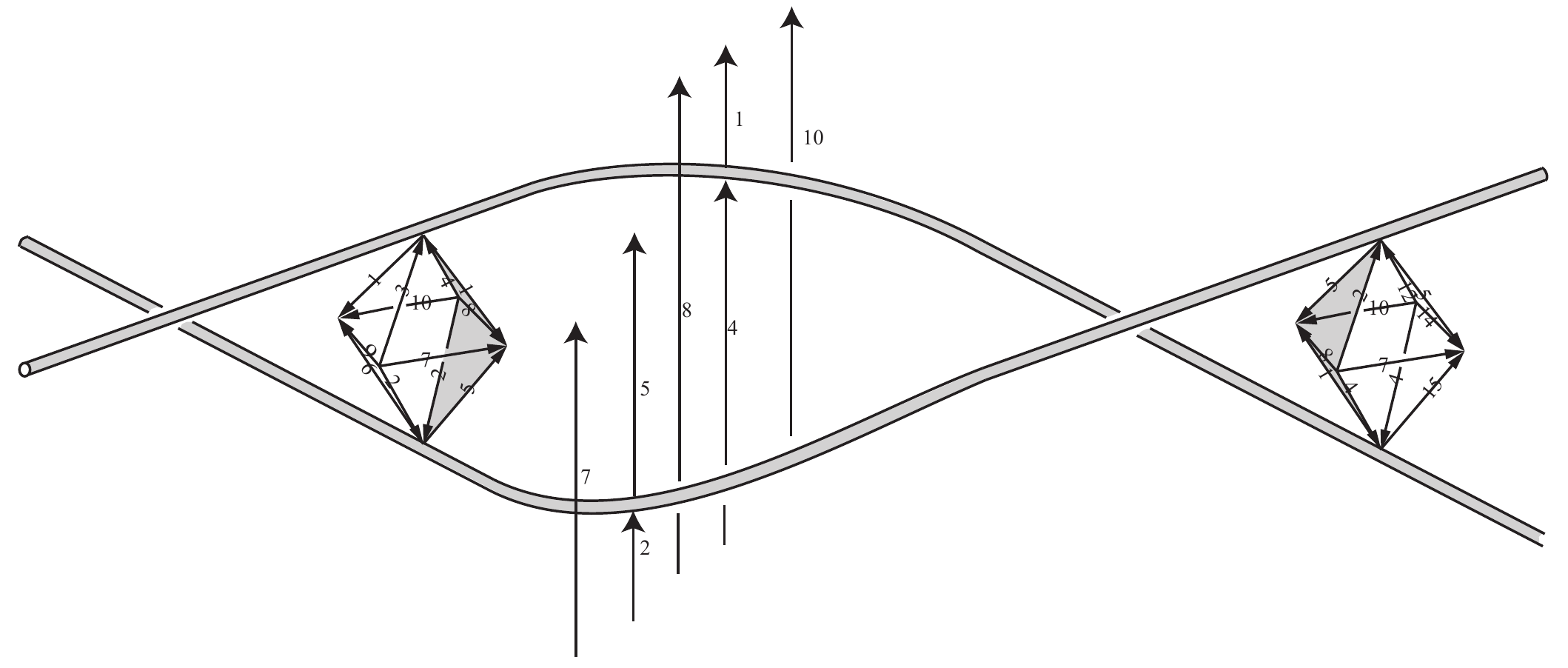}
\caption{Octahedra at the crossings of a bigon.}
\label{bigonoctahedra}
\end{center}
\end{figure}

Note that in the case of a maximal bigon chain of length $n$, if the Thurston construction is applied, there will be $n$ octahedra, and the contribution to the volume from that sequence of crossings will be on the order of $n (3.6638\dots)$. The volume associated with the corresponding $(2n+2)$-bipyramid will be substantially less. For instance, for $n = 499$, the octahedra produce a volume bound of  1828.2, whereas the corresponding1000-bipyramid yields a volume bound of 38.1.

On the other hand, in the Agol-Thurston method, we can augment each bigon chain with an additional link component, and obtain a bound from the bigon chain of just $10 v_{tet}$. Thus, in order to minimize the volume bound, it makes sense to use a $(2n+2)$-bipyramid when its volume is less than $10 v_{tet}$ and to augment with an additional component when its volume is greater than $10v_{tet}$. Hence, by considering the table of volumes of $n$-bipyramids, we see that we should use use the $(2n+2)$-bipyramid for maximal bigon chains of crossing length $n \leq 4$ and use the augmenting component for all chains of crossing length $n > 4$. However, to do so, we need to know that the two constructions can be fit together.

\begin{thm}
\label{thm:bigonvolume} Let $L$ be a  hyperbolic alternating link with five or more crossings in a reduced twist reduced alternating diagram that has at least three twist regions and that is not the Borromean rings.  Then $\mbox{vol}(S^3-L) < t_1 v_{oct} +t_2(6 v_{tet}) + t_3 (7.8549\dots) + t_4(9.2375\dots) + g_5 (10 v_{tet}) - a$, where if $g_2 = 0$, then $a= 15.4972$, if $g_3 = 0$ but $t_2 \ge 1$, then $a = 11 v_{tet}$, if $g_4 = 0$ but $t_3 \ge 1$, then $ a =10.088$,  if $g_5 = 0$ but $t_4 \ge 1$, then, $a = 10.2873$ and if $g_5\ge 1$, then $a = 12.111$. 
\end{thm}

\begin{proof} We begin by inserting one octahedron per crossing, as in the original construction of Dylan Thurston.
 For each maximal bigon chain of crossing length $n$ less than 5, we consolidate the corresponding octahedra at each crossing into a single $(2n+2)$-bipyramid, as described previously. For any maximal bigon chain that has crossing length 5 or greater, we add a trivial component that wraps once around the bigon chain, which we call a vertical component. The resulting link $L'$ is hyperbolic by work in \cite{Adams1}, and its volume is greater than the original link complement by work of W. Thurston. From \cite{Adams2}, it follows that we can cut the link complement open along the twice-punctured disk bounded by the new component, and twist to change the number of crossings in the chain from whatever it was to 2, and the resulting link $L''$ will have complement, possibly homeomorphic to the complement of $L'$ and possibly not, with the same volume as the complement of $L'$. 
 
 We now consider the pairs of octahedra corresponding to the two crossings that were created by the augmentation process. For each such pair, we first glue them together along two faces as described above to obtain a 6-bipyramid (see Figure \ref{drilledpair}(a) and (b)). We  then drill a vertical component out of the complement as in Figure \ref{drilledpair}(a) by drilling it out of this 6-bipyramid, as in Figure \ref{drilledpair}(b). By adding edges to the faces of the 6-bipyramid and collapsing the new drilled arcs down to vertices, we obtain an ideal polyhedron as in Figure \ref{drilledpair}(c). This decomposes into four tetrahedra and one 6-bipyramid, as in Figure \ref{drilledpair}(d). Hence, its contribution to the total volume is at most $10v_{tet}$.
 
 Note that the drilling avoids the faces of the 6-bipyramid that are glued to the faces of other bipyramids corresponding to other bigon chains, so gluing faces does yield a manifold homeomorphic to the complement of $L''$.
 
 Finally, we can collapse the finite vertices $U$ and $D$ to the cusp, by choosing edges from each of them to one of the cusps and shrinking the edges away. 
 
 In the case that $g_2 = 0$, so all bigon chains have crossing length 1, We first assume that the projection is not the alternating projection of the Borromean rings. Then there exists an edge in the projection that bounds two regions, at least one of which has more than three edges. As in the proof of Theorem 5.1 in \cite{Adams5}, we collapse a vertical edge reaching from this strand to U and another from this strand to D, as in edges 1 and 2 in Figure \ref{collapse}. This flattens the two octahedra at the crossings on the ends of this strand.
 
 \begin{figure}[h]
\begin{center}
\includegraphics[scale=0.7]{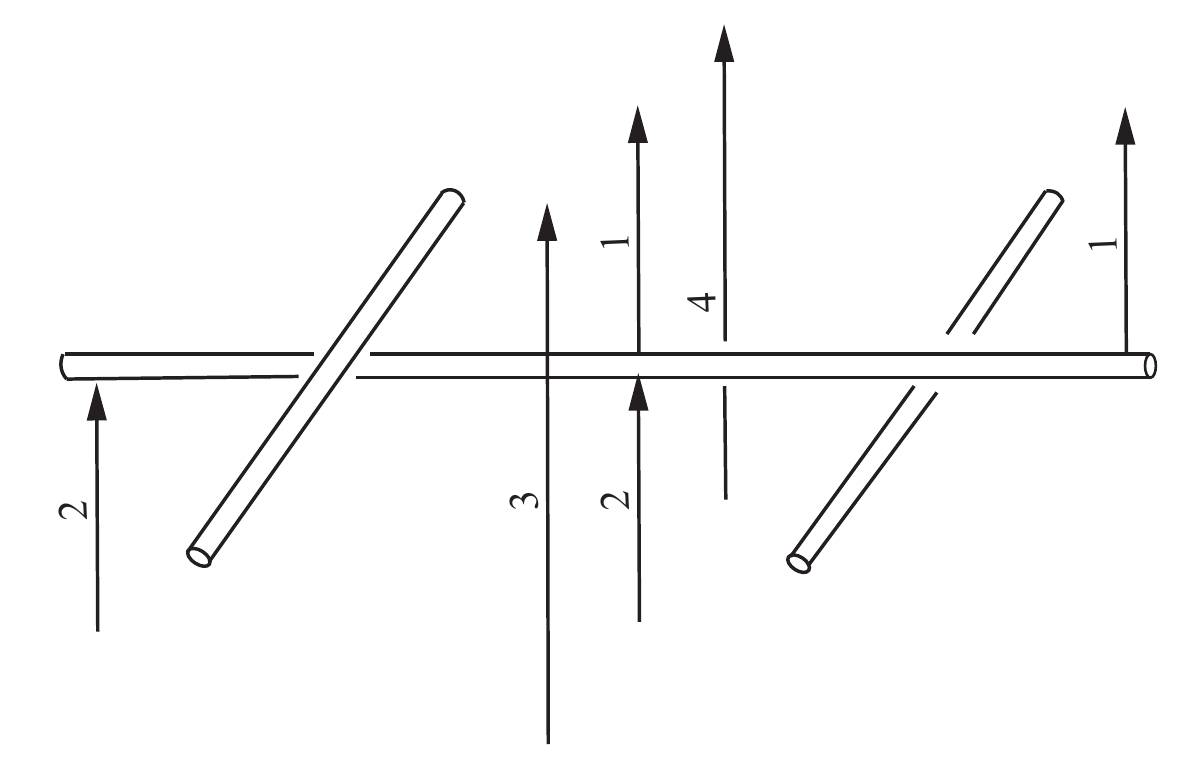}
\caption{Collapsing the edges labelled 1 and 2 causes the edges labelled 3 and 4 to collapse, eliminating the volume contribution of two octahedra and shrinking the volume contribution of another five octahedra.}
\label{collapse}
\end{center}
\end{figure} 

The collapse of the edges labelled 1 and 2 also causes the edges labelled 3 and 4 to collapse. This collapses each of the at least three octahedra corresponding to additional crossings on these regions to two tetrahedra each. It is also true that the edges labelled 1 and 2 can be slid out past the two crossings so that the octahedron to the left and the octahedron to the right each contain one of the edges that is collapsed.  Because there are no bigons, all of these octahedra are distinct. Thus, the total volume drop is $7 v_{oct} - 10 v_{tet} = 15.4972 \dots$.  
 
 When $g_3 = 0$, but $t_2 \ge 0$, we have a 6-bipyramid corresponding to a bigon chain $B$ of crossing length 2. The edges around its equator fall into four equivalence classes, two of size two and two of size one, each edge class of which connects the vertex D to the vertex U. We  collapse one of the edges  with an equivalence class of two to collapse the 6-bipyramid to a 4-bipyramid. This edge appears as the central edge of an adjacent rgion inthe projection plane and it identifies U to D.We also  collapse an edge from U to an ideal vertex. In Figure\ref{drilledpair}, we can realize these collapses by collapsing edge 7 and edge 5. This forces edge 2 to collapse as well and in fact collapses the entire face bounded by these three edges, depicted in grey at the top of Figure \ref{collapsenew}. This ultimately entirely flattens the original 6-bipyramid into two triangular faces with no volume. We also see a volume drop from any bigon chain bipyramid with a crossing in region $A$ that has its central edge collapsed. Because the diagram is twist reduced and has at least three bigon chains, one of the two adjacent regions to $B$ has at least two additional crossings on it.  We choose this to be $A$. The least volume drop comes from when there are exactly two additional crossings and they both occur in another bigon. In this case, appearing as in the projection of the link $6_3^2$ we collapse an additional 6-bipyramid to a 4-bipyramid. There are also two collapsed edge that slide out past the original bigon region $B$. Their smallest impact occurs when the other adjacent region to $B$ has only one crossing. Then these two collapsed edges turn its octahedron into a single tetrahedron.  Thus, the total volume drop is at least $2\mbox{vol}(B_6) + \mbox{vol}(B_4) - (\mbox{vol}(B_4) + v_{tet}) = 11 v_{tet}$.
 
 \begin{figure}[h]
\begin{center}
\includegraphics[scale=0.6]{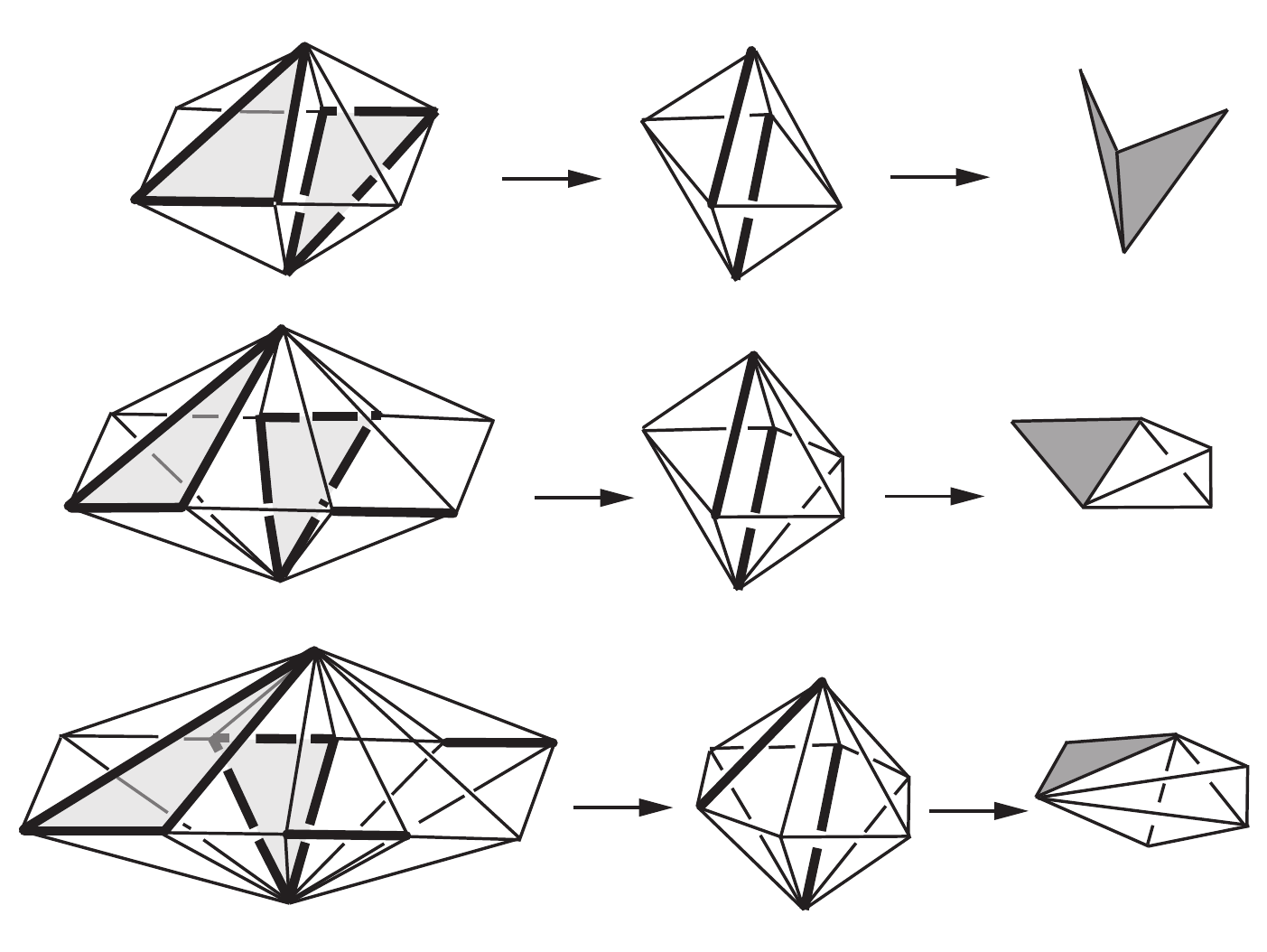}
\caption{Collapsing U and D to a cusp.}
\label{collapsenew}
\end{center}
\end{figure} 
 
 When $g_4 = 0$, but $t_3 \ge 0$, we have an 8-bipyramid corresponding to a bigon chain of crossing length 3. The edges around its equator fall into four equivalence classes, two of size three and two of size one. We  collapse one of the edges corresponding to an adjacent region $A$ with an equivalence class of three and we collapse an edge from U to an ideal vertex, which causes two faces on its boundary to collapse as in the middle of Figure \ref{collapsenew}. This takes the 8-bipyramid to one tetrahedron with an additional flat face. We also see a volume drop from any bigon chain bipyramid with a crossing in region $A$. Again we can assume there is more than one additional crossing in the region. The least volume drop comes from having two crossings on this region at the end of two bigon chains of length 3. However, then two bigon chains of length 3 are adjacent. We do better by taking a bigon chain of length 3 adjacent to a bigon chain of length 2. Then the 6-bipyramid corresponding to the bigon chain of length 2 collapses to a 4-bipyramid.  There is also an edge that slides out of this region with its minimal decrease on volume taking another 8-bipyramid to a 7-bipyramid. This corresponds to the case of the $8_6$ knot projection for instance.  Thus, the total volume drop is at least $2\mbox{vol}(B_8) + \mbox{vol}(B_6) -v_{tet} - \mbox{vol}(B_4) - \mbox{vol}(B_7) = 10.088 \dots$.

 When $g_5 = 0$, but $t_4 \ge 0$, we have a 10-bipyramid corresponding to a bigon chain of crossing length 4. The edges around its equator fall into four equivalence classes, two of size four and two of size one. We first collapse one of the edges with an equivalence class of four. This edge is the central edge perpendicular to an adjacent region $A$ in the projection plane connecting U to D. We also collapse an edge from U to an ideal vertex, causing two faces on its boundary to collapse as in the bottom of Figure \ref{collapsenew}.This collapses the 10-bipyramid to two tetrahedra and a flat face. We also see a volume drop from any bigon chain bipyramid with a crossing on the region $A$. Once again we can assume an adjacent region $A$ has two or more additional creossings on it. The least volume drop come from a bigon having both crossings on the region, in which case the corresponding 6-bipyramid collapses to two tetrahedra. The extra collapsed edge can slide out of the region and the minimal volume drop from it occurs whenit takes a 10-bipyramid to a 9-bipyramid. Thus, the total volume drop is at least $2\mbox{vol}(B_{10}) + \mbox{vol}(B_6) - 2v_{tet} - \mbox{vol}(B_4) - \mbox{vol}(B_9) \approx 10.2873$.

When $g_5 > 0$, we take the 6-bipyramid and its four associated tetrahedra as in Figure \ref{drilledpair}(c) and (d) corresponding to one of the vertical components that were added, and we collapse the edge  g created at the top of edge 8 and the edge h created at the bottom of edge 8.  There are 6 edges on the 6-bipyramid and the four associated tetrahedra in the equivalence class of each of these edges, and their collapse flattens all four tetrahedra and the entire 6-bipyramid. It also forces the collapse of edges 7 and 10. Since edges 7 and 10 are central edges to two regions adjacent to the new vertical component, their collapse will impact any bipyramids that touch these regions. The minimal volume drop comes from the case of having just one additional crossing on the first region, which is the end of a bigon chain of crossing length 4. Then the collapse of the central edge of this region will collapse a 10-bipyramid to a 9-bipyramid and drop volume by  0.6539. The other adjacent region must have more than one other crossing in it as otherwise the projection would not be twist reduced. Considering the various options, one finds that the smallest volume drop occurs when there are two bigon chains of crossing length 4 ending on this region. Then the collapse of the central edge of this region results in collapsing two 10-bipyramids to 9-bipyramids, and yields an additional volume drop of 2(0.6539). Thus, the total volume drop is at least 12.1111.
  \end{proof}
  
  Note that the excluded cases of the Borromean rings and of all reduced twist reduced alternating diagrams with just two twist regions, which all come from Dehn filling the Borromean rings, must have volume bounded above by $2 v_{oct}$.

\begin{figure}[h]
\begin{center}
\includegraphics[scale=0.6]{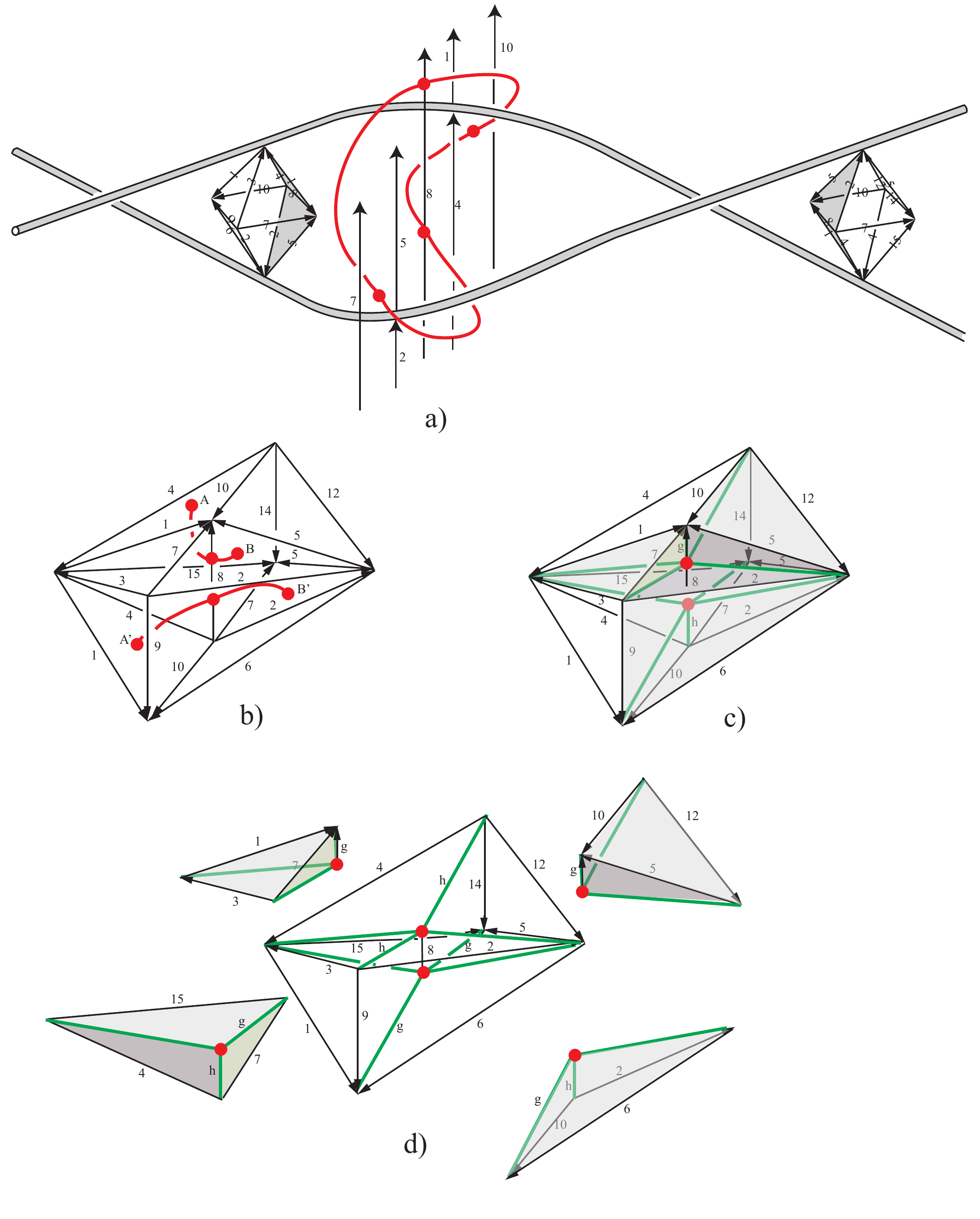}
\caption{Drilling a vertical component out of two paired octahedra.}
\label{drilledpair}
\end{center}
\end{figure} 

As in \cite{DT}, this upper bound on volume can be translated into a result in terms of coefficients of the colored Jones polynomial.
Let $$J_K(n) = \pm (a_n q^{k_n} - b_n q^{k_n -1} + c_n q^{k_n -2}) + \dots \pm (\gamma_n q^{k_n - r_n + 2} - \beta_n q^{k_n-r_n+1} + \alpha_n q^{k_n - r_n})$$ be the colored Jones polynomial of $K$, where $a_n$ and  $\alpha_n$ are positive.

\begin{thm} If $L$ is a hyperbolic alternating link with at least three twist sequences in a twist reduced diagram other than the Borromean rings, then $\mbox{vol}(S^3-L) < (10 v_{tet} - v_{oct}) ((c_2 + \gamma_2) - (c_3+\gamma_3)) - (10 v_{tet} -2v_{oct})(b_2 + \beta_2) - a$, where $ a = 10.088$ if  $ b_2 + \beta_2 \ne (c_2-c_3) + (\gamma_2-\gamma_3)$ and $a = 15.4972$ otherwise.
\end{thm}

\section{Face-Centered Bipyramids}

We introduce  a second method for decomposing a link complement into polyhedra. Given a reduced projection of a knot or link, thought of as a projection on a sphere, let $b_i$ be the number of complementary regions of the projection with exactly $i$ edges. By Euler characteristic, it must be the case that:

$$2b_2 + b_3 =8+  b_5 + 2b_6 + 3b_7 + \dots$$

Note that the number of regions with four edges is not restricted by this formula.

There will be one $n$-bipyramid corresponding to  each complementary face with $n$ edges in the projection.To show this, we begin with the octahedral construction, placing one octahedron at each crossing with top vertex on the underside of the overcrossing and bottom vertex on the top of the undercrossing. Now we cut each octahedron into four tetrahedra along the central vertical edge, as in Figure \ref{octahedroncut}.

\begin{figure}[h]
\begin{center}
\includegraphics[scale=0.7]{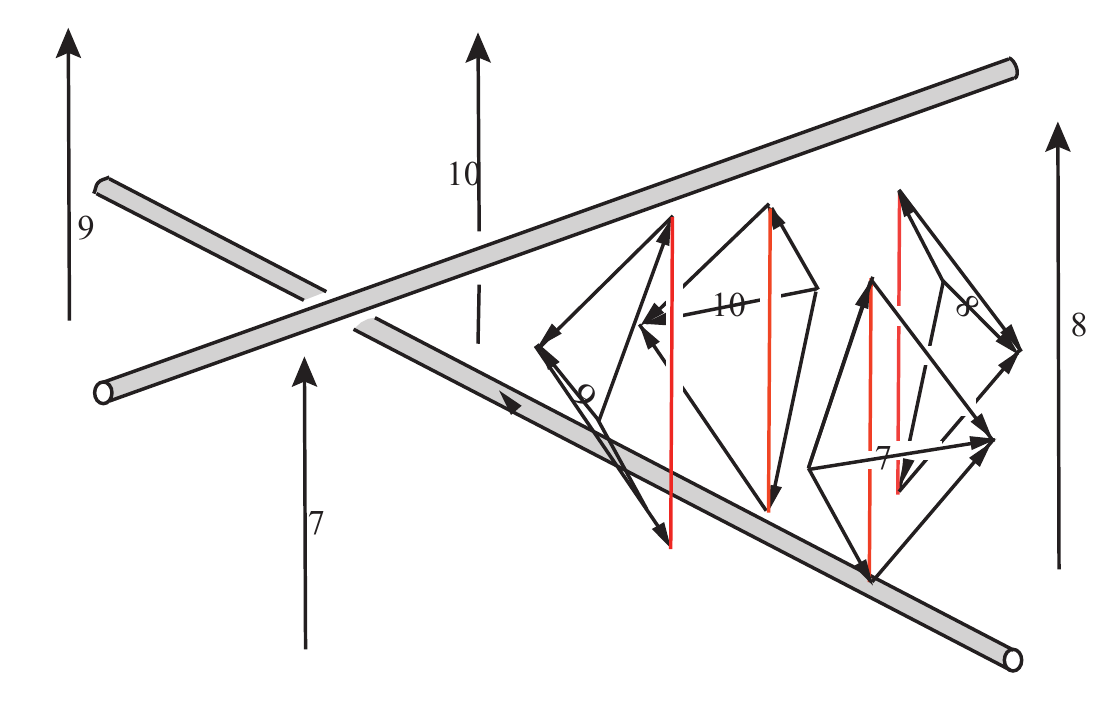}
\caption{Cutting up octahedra and reassembling them into bipyramids.}
\label{octahedroncut}
\end{center}
\end{figure}

 Once we do so, we can glue the resulting tetrahedra together along an edge perpendicular to and at the center of each complementary face. The result is a decomposition of the link complement into bipyramids, each with $n$ ideal vertices and two finite vertices. Note that if a face is a bigon, the resulting bipyramid is a 2-bipyramid, which has no volume. These ``flat" bipyramids can be glued to the adjacent bipyramids resulting in no volume contribution but rather just a change in the gluing of the faces involved.
 
 Since any $n$-bipyramid has volume no more than the ideal regular bipyramid, we see immediately that $$\mbox{vol}(S^3-L)\le \sum_{i=1}^{r} b_i\mbox{vol}(B_n)$$

 We have not yet taken advantage of the fact that the two vertices $U$ and $D$ are finite vertices. But before we drill or collapse to reduce volume, it should be noted that with those finite vertices in place,  in the case of a knot, the collection of faces on the bottom halves of the bipyramids, which all share the vertex denoted $D$ together generate the singular punctured disk that appears in  \cite{H} and subsequently \cite{Adams4}, where it was used to obtain bounds on cusp volumes. This immersed punctured disk is obtained by coning from the knot to a point beneath the projection plane, the role of which is played by $D$. Similarly, the top faces of the bipyramids form a second singular punctured disk which is obtained by coning from the knot to a point above the projection plane corresponding to the point $U$. Put the other way around, one can take any knot (or link) in an alternating projection and cone to a point above and below the projection. Then cutting the complement open along these singular disks yields the decomposition into face-centered bipyramids. 
 
 We also note that the ideal equatorial $n$-gons for each bipyramid glue together to form the union of the two checkerboard surfaces for the projection.
 
We now consider means to eliminate the finite vertices and reduce volume. We can choose the two bipyramids of largest $n$ that are not adjacent to each other, meaning the faces of the projection that generate them do not share an edge. We then drill out an additional link component $C$ that is obtained by removing the union of the two vertical central edges for the pair of bipyramids. If the result is a hyperbolic link, then the volume of the complement of our original link will be less than the volume of the complement of this link, which is at most the sum of the volumes of the ideal regular $n$-bipyramids corresponding to the complementary regions other than the two we have drilled through. The drilling collapses their central edges to ideal points, and thereby eliminates their volume contribution.

In the case that the original link is alternating in a reduced alternating diagram, the resulting link is a generalized augmented alternating link and by Theorem 2.1  of \cite{Adams2}, it is hyperbolic.

A second operation that we can do to lower the volume bound other than drilling  is collapsing. In collapsing, we choose an edge running from the cusp to $D$ and an edge running from the cusp to $U$. Then we collapse both of these edges, pulling the finite vertices to the cusp. To have the maximal effect on volume, we choose our edges to be along a strand of the original link complement that bounds two faces with as many edges as possible. See Figure \ref{collapse}.

The collapse of the edges labelled 1 and 2 causes the edges labelled 3 and 4 to collapse as well. Those two edges are central to the two bipyramids corresponding to the two faces to either side, so their volume is eliminated. It is also true that the edges labelled 1 and 2 can be slid out past the two crossings so that the two bipyramids to the left and the two bipyramids to the right each contain one of the edges that is collapsed. As in Figure \ref{collapse2}, for $n \ge 3$, such a collapse on an $n$-bipyramid results in a collection of $n-2$ tetrahedra. However, all of the resulting tetrahedra share an edge. By adding one more tetrahedron, we can complete this collection to obtain an $(n-1)-$ bipyramid. This collection of tetrahedras can therefore have no more volume than the resulting $(n-1)$-bipyramid. In the case $n \ge 12$, the resulting $(n-1)$-bipyramid has volume less than $n-1$ regular ideal tetrahedra. But in the case $3 \le n \le 11$, the $n-1$ tetrahedra have volume less than the $(n-1)$-bipyramid.   Thus we find that the volume is bounded by the sum of the volumes of all of the bipyramids, except for the two adjacent ones that we have collapsed and the four others that have each been reduced from a contribution of $\mbox{vol}(B_n)$ to a contribution of either $(n-2) v_{tet}$ or $\mbox{vol}(B_{n-1})$. Thus, we have the following theorem.

\begin{figure}[h]
\begin{center}
\includegraphics[scale=0.7]{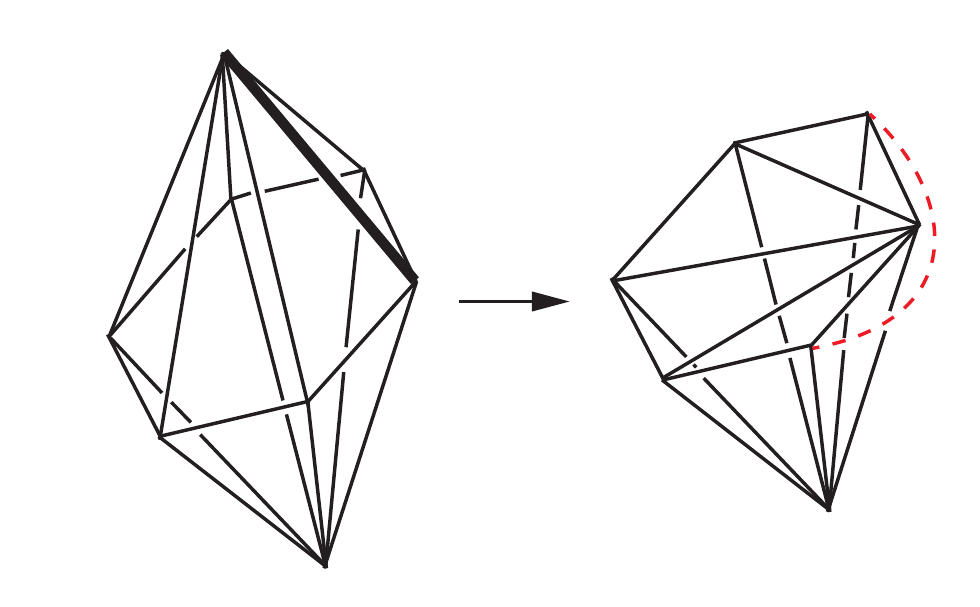}
\caption{Collapsing the thickened edge on an $n$-bipyramid  yields $n-2$ tetrahedra that all share an edge. Adding the dashed edge turns this into an $(n-1)$-bipyramid.}
\label{collapse2}
\end{center}
\end{figure} 

\begin{thm} If $L$ is a hyperbolic alternating link with reduced alternating projection with $b_i$ faces of $i$ edges, then $$\mbox{vol}(S^3 - L) \le \sum b_i \mbox{vol}(B_i) - a$$ where $a = max\{r,s\}$, with $r$ the maximum value that can be attained when $r = \mbox{vol}(B_c) + \mbox{vol}(B_d)$ for $c$ and $d$ the numbers of edges for a pair of nonadjacent faces, and $s$ the maximum value that can be attained when $$s = \mbox{vol}(B_c) + \mbox{vol}(B_d) + \sum_{j=1}^4 \mbox{vol}(B_{e_j})  - \sum_{j=1}^4 V_{e_j}$$ where $c$ and $d$ are the numbers of edges in a pair of adjacent faces $F_c$ and $F_d$, and $\{e_j\}$ are the edge numbers for the four faces that each share a crossing with both of $F_c$ and $F_d$. When $e_j \ge 12$, $V_{e_j} = \mbox{vol}(B_{e_j-1})$. When $3 \le e_j  \le 11$, $V_{e_j}= (n-2) v_{tet}$.
\end{thm}

\begin{ex} We consider the figure-eight knot. Its standard projection consists of complementary regions with four triangles and two bigons. Choosing an edge between two triangles, we collapse away the two 3-bipyramids to either side of the edge. At each end of the edge, there is also a 2-bipyramid and a 3-bipyramid. Collapsing an edge on the 2-bipyramid does nothing since it already contributes zero volume. Collapsing an edge  on the  3-bipyramid yields a singe tetrahedron. So we know that the complement of the figure-eight knot has a volume of at most $2 v_{tet}$. This is in fact the actual volume of the complement of the figure-eight knot, demonstrating that there is at least one knot for which this upper bound on volume is exact.
\end{ex}

Note  that in the case of an alternating hyperbolic link,  between drilling and collapsing of face-centered bipyramids, we know that we can eliminate the volume contribution of any two bipyramids, whether they are adjacent or not.

Of course, we would like to do better. We would like to be able to drill out additional link components to eliminate other face-centered bipyramids of high $n$. 

If we have two faces that share a bigon chain, and we would like to remove the bigon chain by drilling out a vertical component, as in the Agol-Thurston construction,  we use the method depicted in Figure \ref{drilledpair}. That is, we replace the bigon chain by a bigon with two crossings, we insert two octahedra at the crossings, we glue them together and drill out the vertical component from them, and then decompose them into four tetrahedra and one 6-bipyramid. Then, as described previously, we take all other crossings and fill them with octahedra, and then decompose the other octahedra into four tetrahedra. We then glue the resulting tetrahedra together along the vertical edges through the centers of all faces. When it comes to the two crossings along our bigon, we glue the four individual tetrahedra in to the bipyramids corresponding to  the two faces adjacent to the bigon. The resulting $n$ for each $n$-bipyramid that results has  $n$-value corresponding to the link diagram without the extra vertical component but with the inserted bigon.

If we begin with a bigon chain of crossing length $k$, then by drilling the corresponding vertical component, we switch two bipyramids corresponding to the adjacent faces from $r$ and $s$ to $r-k+2$ and $s-k+2$. However, we have also replaced the chain of 2-bipyramids given by the chain of bigons, all of which had zero volume, by a single 6-bipyramid.  So it only makes sense to drill when $\mbox{vol}(B_r) + \mbox{vol}(B_s) > \mbox{vol}(B_{r-k+2}) + \mbox{vol}(B_{s-k+2}) + 6.089$.

\subsection{Flyping}

Given a particular reduced alternating projection of a hyperbolic link, we can use face-centered bipyramids to obtain an upper bound on the volume of the link complement. One would like to know that the bound obtained is as low as possible.

\begin{thm} The smallest upper bound on volume obtained from a reduced alternating projection of a hyperbolic link $L$ via face-centered bipyramids occurs in a twist reduced diagram.
\end{thm}

\begin{proof} We show that if a diagram of $L$ is not twist-reduced, the sum of the volumes of the corresponding face-centered bipyramids can be reduced.
Suppose $P$ is a reduced alternating diagram that is not twist-reduced. Then there must be a portion of the diagram with crossings to either side of a tangle as in Figure \ref{flype}. 

\begin{figure}[h]
\begin{center}
\includegraphics[scale=0.7]{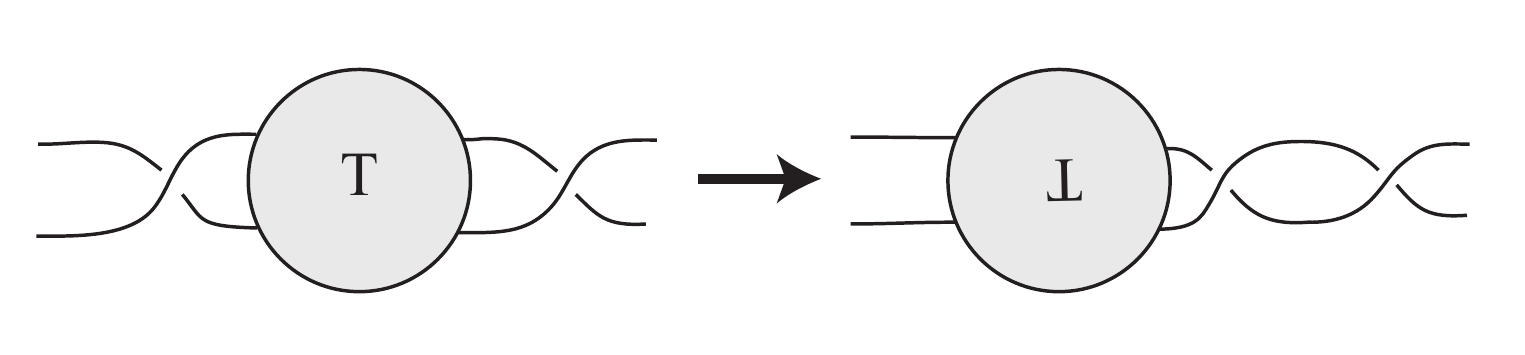}
\caption{Flyping to twist reduce lowers the volume bound.}
\label{flype}
\end{center}
\end{figure} 

If there are bigons to the outer side of both of these crossings, then we can flype to move all but one of the crossings to one side of the tangle, without changing the numbers of $n$-bipyramids for all $n$. However, when we are down to only one crossing remaining to one side, then when we flype to twist-reduce and move it across, we replace an $a$-bipyramid and a $b$-bipyramid with a single $(a+b-2)$-bipyramid. Thus it is enough to show that $\mbox{vol}(B_a) + \mbox{vol}(B_b)   > \mbox{vol}(B_{a+b-2})$ for all $a,b \geq 3$.
Assume that $b \geq a$. Suppose first that $a \geq 6.$ Note that the tetrahedron $T_a$, with angles $\frac{2\pi}{a}, \frac{(a-2)\pi}{2a}$  and $\frac{(a-2)\pi}{2a}$ has volume a decreasing function in $a$ for $a \geq 6$. Then 

\begin{align}\mbox{vol}(B_a) + \mbox{vol}(B_b) &  = a \mbox{vol}(T_a) + b \mbox{vol}(T_b) \\ 
& \geq (a+b)\mbox{vol}(T_b)\\ 
& \geq (a+b-2) \mbox{vol}(T_{a+b-2}) = \mbox{vol}(B_{a+b-2})
\end{align}

In the case that $a = 3$, 4 and 5, one can check that $\mbox{vol}(B_a) > (a-1) (1.01494...)$.

Hence, we have

\begin{align}
\mbox{vol}(B_a) + \mbox{vol}(B_b) &  \geq (a-1)(1.01494...) + b \mbox{vol}(T_b)\\ 
& \geq (a+b-1)\mbox{vol}(T_b)\\
&  \geq (a+b-2) \mbox{vol}(T_{a+b-2}) = \mbox{vol}(B_{a+b-2})
\end{align}

\end{proof}

\section{Comparisons of the upper bounds}

In this section, we compare the bounds we have discussed. Note first that the octahedral upper bound is always better than the tetrahedral upper bound   for $c \ge 6$. Similarly, the BCB bound always improves the DT bound, which itself improved the AT bound.

For low crossing number, the FCB (face-centered bipyramid) bound always seems to do better than the BCB (bigon chain bipyramid) bound. However, if one chooses a link such as the one appearing in 
Figure \ref{pentagonal}, but with $n^2$ hexagonal faces, then we can think of the volume contribution to the BCB bound of each bigonal chain of length 3 weighted against the volume contribution of its two adjacent hexagonal faces tom the FCB bound. We add in the volume contribution of each singleton crossing to the BCB bound and weight it against the volume contribution of the four adjacent faces to the FCB bound. Therefore, the contribution to the FCB bound of a hexagonal face is $\mbox{vol}(B_6) = 6.0896$. The fractional contribution to the BCB bound of the adjacent bigon sequence and singleton crossings is $\mbox{vol}(B_8)/2 + 2(\mbox{vol}(B_4)/4 = 5.7614$. Thus the BCB bound beats the FCB bound by 0.3282 for every additional hexagonal face in the link complement. Although faces on the boundary that have fewer edges than hexagons mediate this effect, their number grows with $n$ rather than $n^2$.  So for large enough $n$, the BCB bound will always be lower than the FCB bound.

\begin{figure}[h]
\begin{center}
\includegraphics[scale=0.7]{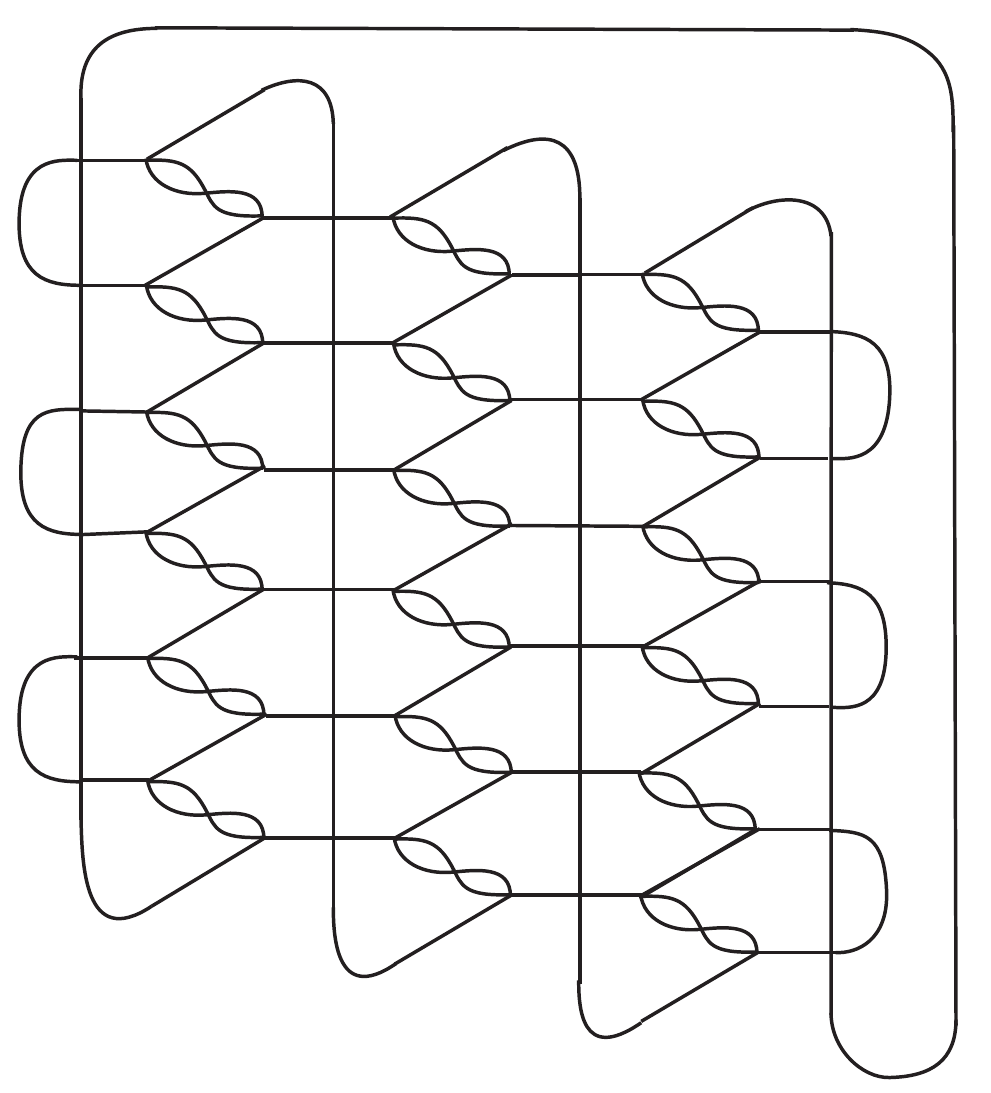}
\caption{For large enough $n$, the BCB bound is lower than the FCB bound.}
\label{pentagonal}
\end{center}
\end{figure} 

However, at least for small crossing number, the FCB bound appears more effective, so we will focus on it. In calculating a volume bound using the FCB bound, one can either drill two non-adjacent bipyramids or one can collapse as described in the discussion of the algorithm for the bound. In the case of certain knots and links, drilling can do better than collapsing, as occurs with $9_{35}$ for instance.  The same holds true for any $(3,3, \dots, 3)$-pretzel link or $(4,4, \dots, 4)$-pretzel link, with at least three 3's or 4's. Other than $4_1$ and $5_2$,  this also holds for the twist knots, where in all cases, drilling yields a bound of $4v_{tet}$, and the sequence of volumes are known to approach $v_{oct} = 3.6638\dots$ from below.

 In the case of $7_1^2$, drilling and collapsing in the FCB bound yield the same value. But more generally, at least for low crossing number, collapsing yields a better bound than does drilling. In the two cases of $4_1$ and $6_2^2$, the FCB bound obtained by collapsing yields the actual volume, namely $2v_{tet}$ and $4v_{tet}$ respectively. For hyperbolic knots of seven or fewer crossings, the FCB bound is within 4.1\% of the actual volume for $7_5$, within 10.6\% for the additional knots $5_2, 6_3, 7_3, 7_5, 7_6$ and $7_7$. Excluding the twist knots $6_1$ and $7_2$, all the hyperbolic knots of seven or fewer crossings have FCB bound within 19\% of the actual volume.
 
 In the case of $5_1^2$, both the FCB bound and the octahedral bound yield the same bound of $4v_{tet}$, whereas the actual volume is $v_{oct}$.  But in general, the FCB bound beats the octahedral bound.

\end{document}